\documentclass[a4paper, 12pt]{amsproc}

\usepackage{fullpage}
\usepackage{amsmath, amsthm, amssymb, mathtools, mathrsfs}
\usepackage{ascmac}
\usepackage{comment}
\usepackage{bm}

\usepackage{hyperref}

\usepackage{graphicx}
\usepackage{here}
\usepackage{time}
\usepackage[abbrev]{amsrefs}

\usepackage{xcolor}
\usepackage[capitalize,nameinlink,noabbrev,nosort]{cleveref}
\hypersetup{
	colorlinks=true,       
	linkcolor=blue,          
	citecolor=blue,        
	filecolor=blue,      
	urlcolor=blue,           
}

\makeatletter
\@namedef{subjclassname@2020}{%
  \textup{2020} Mathematics Subject Classification}
\makeatother

\newtheorem{theoremcounter}{Theorem Counter}[section]

\theoremstyle{definition}
\newtheorem{defi}[theoremcounter]{Definition}
\newtheorem{rem}[theoremcounter]{Remark}
\newtheorem{ex}[theoremcounter]{Example}

\theoremstyle{plain}
\newtheorem{lem}[theoremcounter]{Lemma}
\newtheorem{prop}[theoremcounter]{Proposition}
\newtheorem{cor}[theoremcounter]{Corollary}
\newtheorem{conj}[theoremcounter]{Conjecture}
\newtheorem{thm}[theoremcounter]{Theorem}

\numberwithin{equation}{section}\numberwithin{figure}{section}

\newcommand{\bbH}{\mathbb{H}}

\DeclareMathOperator{\ImNew}{Im}
\renewcommand{\Im}{\ImNew}
\DeclareMathOperator{\ReNew}{Re}
\renewcommand{\Re}{\ReNew}

\DeclareMathOperator{\SL}{SL}
\DeclareMathOperator{\sgn}{sgn}
\DeclareMathOperator{\erf}{erf}

\def\st#1#2{\genfrac{[}{]}{0pt}{}{#1}{#2}}
\def\sm{{M (p_1/q_1,\dots,p_n/q_n)}}

\theoremstyle{thmstyleone}%
\theoremstyle{thmstyletwo}%

\theoremstyle{thmstylethree}%

\begin{document}

\title[Modular transformation formulas of homological blocks]{Modular transformations of homological blocks for Seifert fibered homology $3$-spheres}

\author{Toshiki Matsusaka}
\address{Faculty of Mathematics, Kyushu University, Motooka 744, Nishi-ku, Fukuoka 819-0395, Japan}
\email{matsusaka@math.kyushu-u.ac.jp}

\author{Yuji Terashima}
\address{Graduate school of science, Tohoku University, 6-3, Aoba, Aramaki-aza, Aoba-ku, Sendai 980-8578, Japan}
\email{yujiterashima@tohoku.ac.jp}


\maketitle

\begin{abstract}
In this article, for any Seifert fibered integral homology 3-sphere, we give explicit modular transformation formulas of homological blocks introduced by Gukov-Pei-Putrov-Vafa. Moreover, based on the modular transformation formulas, we have explicit asymptotic expansion formulas for the Witten-Reshetikhin-Turaev invariants which give a new proof of a version by Andersen of the Witten asymptotic conjecture.
\end{abstract}


\section{Introduction}

Recently, Gukov-Pei-Putrov-Vafa~\cite{GPPV2020} introduced important $q$-series called homological blocks for any plumbed $3$-manifolds associated with negative definite plumbing tree graphs based on Gukov-Putrov-Vafa \cite{GPV}.
A physical viewpoint strongly suggests that the homological blocks have several interesting properties \cite{CCFGH2019, ChengFerrariSgroi2020, Chun2017, Chung2020, EGGKPS2020,GHNPPS2021, GukovManolescu2021, BringmannMahlburgMilas2019, BringmannMahlburgMilas2020, BKMN2021-pre, BKMN2021-RMS, MoriMurakami2021}. In particular, it is expected that the homological blocks have good modular transformation properties and their special limits at root of unity 
are identified with the Witten-Reshetikhin-Turaev (WRT) invariants. These are interesting mathematical conjectures. For Seifert fibered integral homology $3$-spheres, the identification of their special limits with the WRT invariants is shown  in~\cite{FIMT2021} and~\cite{AndersenMistegard2018}. In fact, for
Seifert fibered integral homology $3$-spheres, Fuji, Iwaki, Murakami, and the second author~\cite{FIMT2021} introduced a $q$-series
called the WRT function which is identified with a homological block by Andersen-Misteg{\aa}rd~\cite{AndersenMistegard2018} 
and proved that WRT invariants for Seifert fibered integral homology spheres are radial limits at root of unity of the WRT
functions. Independently in a different way, Andersen-Misteg{\aa}rd~\cite{AndersenMistegard2018} proved the same radial limit formula for the homological blocks.

In this article, we get explicit modular transformation formulas of homological blocks for any Seifert fibered integral homology $3$-sphere (\cref{thm:S-trans-WRT}). 
Moreover, combined with the results in~\cite{FIMT2021} and~\cite{AndersenMistegard2018}, we have explicit asymptotic expansion formulas for the WRT invariants which give a new proof of a version by Andersen \cite{Andersen2013} of the Witten asymptotic conjecture \cite{Witten1989} of the WRT invariants (\cref{thm:asymp-Andersen} and \cref{CS}). 
The Witten asymptotic conjecture for Seifert $3$-manifolds is studied in different forms with different methods in~\cite{FreedGompf1991, Jeffrey1992, Rozansky1995, Rozansky1996, Hansen2005, HansenTakata2002, Hikami2005IJM, Hikami2005IMRN, Hikami2006, BeasleyWitten2005, Andersen2013, AndersenHimpel2012, Charles2016}. 
We remark that an asymptotic formula in~\cite[Theorem 16]{Hikami2006} obtained by using a combination of modular transformations of Eichler integrals and their derivatives differs from \cref{thm:asymp-Andersen} in this article, and conjecturally coincides.
 
Our main tools are modular transformation formulas for generalized false theta functions 
based on an idea in Bringmann-Nazaroglu~\cite{BringmannNazaroglu2019}. False theta functions are functions that are similar to the ordinary theta functions but have quite different behavior. 
It has a long history, and the name ``false theta functions" already appears in Ramanujan's last letter to Hardy in 1920. 
However, their modular aspects had long remained a mystery. In 2019, Bringmann and Nazaroglu succeed in finding modular completions for a certain class of false theta functions parallel with Zwegers' results~\cite{Zwegers2002} for mock theta functions. 
Following their results, we capture the homological blocks in the framework of modular forms. 

This article is organized as follows. In \cref{sec:hb}, we prepare the settings for the Seifert fibered integral homology $3$-sphere and the homological blocks. In \cref{sec:ftf}, we introduce ordinary theta functions and false theta functions and show their modular $S$-transformations. In \cref{sec:de-ftf}, we show the expression of the homological block in terms of false theta functions. Combining them, we obtain our main theorem on modular transformation formulas of the homological blocks. 
As an application of our modular transformation formulas, we have explicit
asymptotic expansion formulas of the WRT invariants.

\section{Homological blocks/WRT functions}
\label{sec:hb}
Gukov-Pei-Putrov-Vafa~\cite{GPPV2020} derived homological blocks for any plumbed $3$-manifolds associated with negative definite plumbing tree graphs as integrals based on Gukov-Putrov-Vafa \cite{GPV}. In this section, for any Seifert fibered integral homology $3$-sphere, we define the homological block as a $q$-series which is obtained by expanding the integral and explain a relation to the 
Witten-Reshetikhin-Turaev (WRT) invariant.

We denote the Seifert fibered $3$-manifold with $n$-singular fibers and surgery integers 
$p_1, \dots, p_n$ and $q_1, \dots, q_n$ by $\sm$. 
More precisely, we take pairwise coprime integers $p_1, \dots, p_n \ge 2$
and nonzero integers $q_1, \dots, q_n$ satisfying  
\begin{equation} \label{eq:integral-homology-condition}
p_1 \cdots p_n \sum_{j=1}^{n} \frac{q_j}{p_j} = 1.
\end{equation}
Then, $\sm$ is obtained by a rational surgery along a link 
$L_0 \cup L_1 \cup \cdots \cup L_n$ inside $S^{3}$ depicted in \cref{Figure1}.

\begin{figure}[h]
	\centering
	\includegraphics[width=6cm]{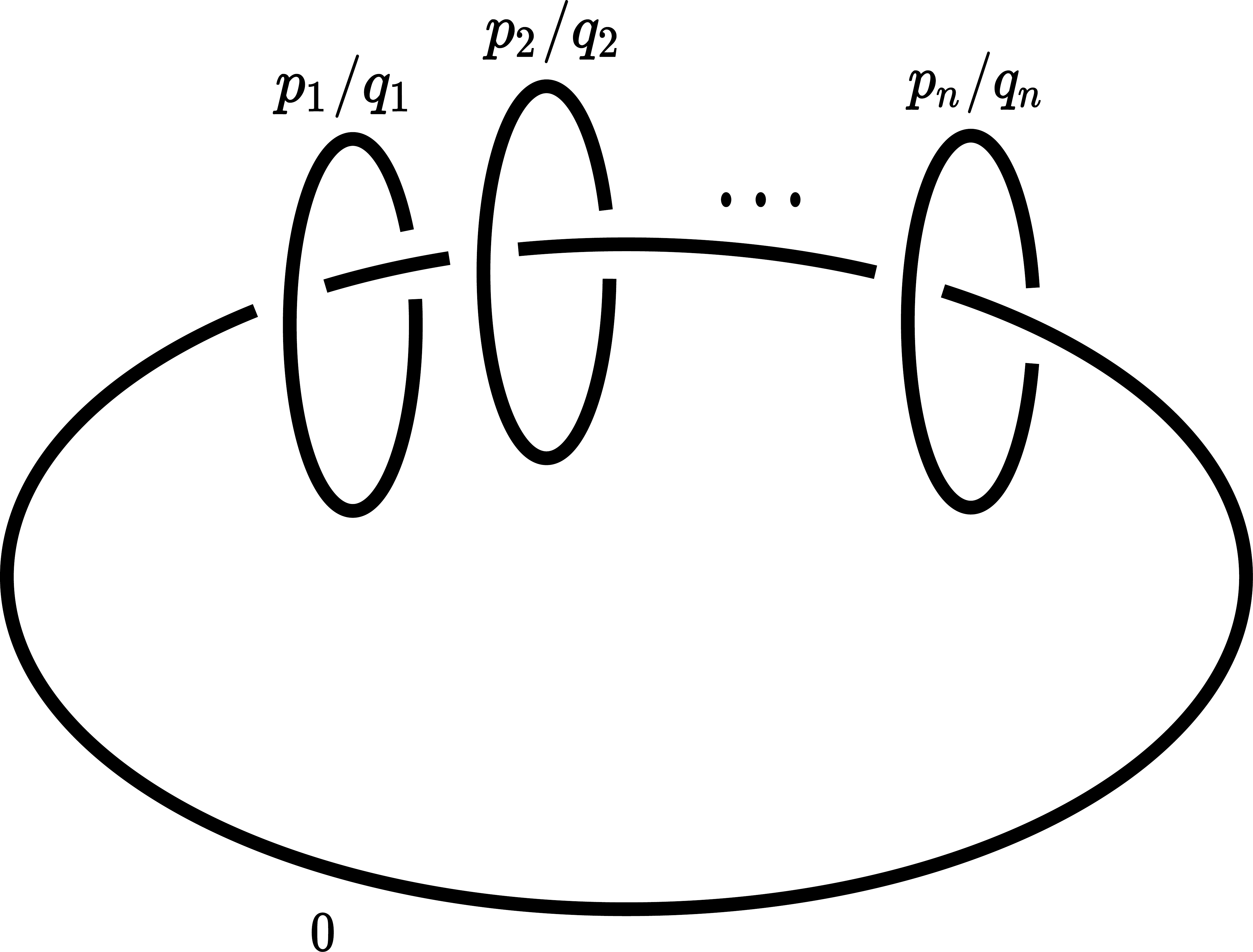}
	\caption{Surgery description of $M(p_1/q_1, \dots, p_n/q_n)$.}
	\label{Figure1}
\end{figure}

Here, the surgery indices of $L_0, L_1, \dots, L_n$ 
are $0, p_1/q_1, \dots, p_n/q_n$, respectively. 
The assumption \eqref{eq:integral-homology-condition} guarantees 
that the Seifert fibered $3$-manifold $M$ is an integral homology $3$-sphere. 
For a Seifert fibered integral homology $3$-sphere $M=\sm$, we define a homological block $\Phi(q)$ as the following $q$-series:

\begin{defi}\label{defi:WRT-function}
\begin{align}
\begin{split}
& \Phi (q) :=\frac{(-1)^n}{2(q^{\frac{1}{2}} - q^{-\frac{1}{2}})} 
\,\, q^{-\frac{1}{4}\Theta_0}\\
&\quad \times 
\sum_{(\varepsilon_1, \dots, \varepsilon_n) \in \{\pm 1 \}^n} 
\varepsilon_1 \cdots \varepsilon_n \,
\sum_{m=0}^{\infty} \dbinom{m+n-3}{n-3} \, 
q^{\frac{P}{4} (2m+ n-2 + 
\sum_{j=1}^{n}\frac{\varepsilon_j}{p_j})^2}. 
\label{eq:def-WRT-function}
\end{split}
\end{align}
Here, we have used the same notations in \cite{LawrenceRozansky1999}:
\begin{align}
P & := p_1 \cdots p_n, \\
\Theta_0 & := 3 - \frac{1}{P} 
+ 12 \sum_{j=1}^{n} s(q_j, p_j), 
\end{align}
where $s(q,p)$ is the Dedekind sum defined by 
\begin{equation}
s(q,p) := \frac{1}{4p} \sum_{\ell=1}^{p-1} \cot\Bigl( \frac{\pi \ell}{p} \Bigr) \, 
\cot \Bigl( \frac{\pi \ell q}{p} \Bigr).
\end{equation}
\end{defi}
This definition coincides with that of the WRT function introduced in \cite{FIMT2021}. It is shown in~\cite{AndersenMistegard2018} 
that the WRT function $\Phi(q)$
is essentially the same as the homological block $\widehat{Z}(q)$ in \cite{GPPV2020} for any Seifert fibered integral homology $3$-sphere. More precisely, 
$$
\Phi(q)=\frac{(-1)^n}{2(q^{\frac{1}{2}} - q^{-\frac{1}{2}})} 
\,\, q^{-\frac{1}{4}\Theta_0+\Delta}\ \widehat{Z}(q)
$$
holds for some explicit rational number $\Delta$. 
The following theorem is proved by Fuji-Iwaki-Murakami-Terashima \cite{FIMT2021} and Andersen-Misteg{\aa}rd~\cite{AndersenMistegard2018} independently in different ways:

\begin{thm}\label{thm:q-series-and-quantum-invariant}
For each $K \in {\mathbb Z}_{\ge 1}$, we have
\begin{equation} \label{eq:q-series-limiting-value}
\lim_{t \rightarrow 0+} \Phi\left(e^{\frac{2\pi i}{K}} \, e^{-t}\right) 
= 
\tau_K,
\end{equation}
where $\tau_K = \tau_K(M)$ is the WRT invariant with level $K$.
\end{thm}

We note that some special cases of \cref{thm:q-series-and-quantum-invariant} 
were proved in previous works.
Lawrence-Zagier \cite{LawrenceZagier1999} prove the statement for the Poincar\'e homology sphere 
(i.e., $n=3$ and $(p_1, p_2, p_3) = (2,3,5)$), 
and Hikami~\cite{Hikami2005IJM} proves for the Brieskorn homology spheres  
(i.e.,  $n=3$ and general pairwise coprime triple $(p_1, p_2, p_3)$). 
\cref{thm:q-series-and-quantum-invariant} suggests that the $q$-series $\Phi(q)$ is a kind of an analytic continuation of 
the WRT invariant $\tau_K$ with respect to $K$ from 
integers to complex numbers.

\section{False theta functions}
\label{sec:ftf}

In this section, we generalize Bringmann-Nazaroglu's results~\cite{BringmannNazaroglu2019} to false theta functions related to the homological blocks. Let $\bbH$ denote the upper half-plane $\bbH = \{\tau \in \mathbb{C} \mid \Im(\tau) > 0\}$ and put $q = e^{2\pi i \tau}$ for $\tau \in \bbH$. For a positive integer $M > 0$, let $L = \sqrt{M} \mathbb{Z}$ be a lattice and $L' = (1/\sqrt{M}) \mathbb{Z}$ be its dual lattice. Throughout this article, we take a branch of the square root satisfying $\arg \sqrt{z} \in (-\pi/2, \pi/2]$. Then ordinary theta functions and false theta functions are defined as follows.

\begin{defi}\label{defi:false-theta}
	For each $k \in \mathbb{Z}_{\geq 0}$ and $\mu \in L'/L$, we define an \emph{ordinary theta function} $\theta_{k, \mu}(\tau)$ and a \emph{false theta function} $\widetilde{\theta}_{k,\mu}(\tau)$ by
	\begin{align*}
		\theta_{k, \mu}(\tau) = \sum_{n \in L + \mu} n^k q^{\frac{n^2}{2}}, \qquad \widetilde{\theta}_{k, \mu}(\tau) = \sum_{n \in L + \mu} \sgn(n) n^k q^{\frac{n^2}{2}},
	\end{align*}
	where $\sgn : \mathbb{R} \to \{-1, 0, 1\}$ is the usual sign function with $\sgn(0) = 0$.
\end{defi}

It is well-known that the ordinary theta functions for $k=0,1$ are classical holomorphic modular forms. As for the false theta functions,
in the case of $k=0$, Bringmann and Nazaroglu showed the $S$-transformation
\begin{align}\label{eq:BN-Strans}
    \begin{split}
	&\widetilde{\theta}_{0,\mu} \left(-\frac{1}{\tau} \right) + \sgn(\Re(\tau)) \frac{(-i\tau)^{1/2}}{\sqrt{M}} \sum_{\nu \in L'/L} e^{2\pi i \mu \nu} \widetilde{\theta}_{0,\nu}(\tau)\\
	&\quad = -i \int_0^{i\infty} \frac{\theta_{1,\mu}(z)}{\sqrt{-i (z+1/\tau)}} dz,
	\end{split}
\end{align}
where $\Re(\tau) \neq 0$. 
In the following subsections, we show the $S$-transformations of the ordinary and (a certain linear combination of) false theta functions for any $k \geq 0$.
The following proposition, so-called Poisson's summation formula, is the most standard tool to show the $S$-transformation of theta functions.

\begin{prop}[Poisson's summation formula]
    Let $f: \mathbb{R} \to \mathbb{C}$ be a Schwartz function. For any $x \in \mathbb{R}$, we have
    \[
        \mathrm{vol}(\mathbb{R}/L) \sum_{n \in L} f(n+x) = \sum_{n \in L'} \mathcal{F}(f)(n) e^{2\pi i nx},
    \]
    where $\mathcal{F}(f)$ is the \emph{Fourier transform} of $f$ defined by
    \[
        \mathcal{F}(f)(x) = \int_{-\infty}^\infty f(y) e^{-2\pi i xy} dy.
    \]
\end{prop}

\subsection{Ordinary theta functions}
\label{sec:Sth}

The next lemma immediately follows.

\begin{lem}\label{Fourier-phi-k}
	For $k = 0,1$, let $\phi_k(\tau, x) = x^k e^{\pi i x^2 \tau}$. Then we have
	\begin{align*}
		\mathcal{F}(\phi_k(\tau, \cdot))(x) = (-i)^{-1/2} \tau^{-1/2-k} \phi_k \left(-\frac{1}{\tau}, x \right).
	\end{align*}
\end{lem}

The following $S$-transformation formulas for ordinary theta functions are shown by this lemma and Poisson's summation formula.

\begin{prop}\label{theta-01-S-trans}
    For $k=0,1$,
    \[
        \theta_{k, \mu} \left(-\frac{1}{\tau} \right) = \frac{(-1)^k (-i)^{1/2}}{\sqrt{M}} \tau^{k+1/2} \sum_{\nu \in L'/L} e^{2\pi i \mu \nu} \theta_{k, \nu}(\tau).
	\]
\end{prop}

We remark that although \cref{Fourier-phi-k} does not hold for $k>1$, it can be generalized for every $k \geq 0$. To be more precise, we introduce the Hermite polynomials $H_n(x)$ defined by
\[
	e^{-2\pi t(2x+t)} = \sum_{n=0}^\infty H_n(x) \frac{t^n}{n!}.
\]
The first few examples are given by $H_0(x) = 1, H_1(x) = -4\pi x, H_2(x) = 16\pi^2 x^2 - 4\pi$. Then we can check that the function
\[
	\mathcal{H}_k(\tau,x) = \Im(\tau)^{-k/2} H_k(x \sqrt{\Im(\tau)}) e^{\pi i x^2 \tau}
\]
satisfies a similar formula
\[
	\mathcal{F}(\mathcal{H}_k(\tau, \cdot))(x) = (-i)^{-1/2} \tau^{-1/2-k} \mathcal{H}_k \left(-\frac{1}{\tau}, x\right).
\]
For the proof, see Vign\'eras~\cite[Lemma 4]{Vigneras1977-proc}.

Now, we generalize \cref{theta-01-S-trans} for $k>1$ in another way. To express the $S$-transformation $\theta_{k,\mu}(-1/\tau)$ in general cases, we need theta functions of different weights.

\begin{prop}\label{theta-S-transform}
	For $k = 2\kappa+\iota$ with $\kappa \in \mathbb{Z}_{\geq 0}$ and $\iota = 0,1$, 
	\[
		\theta_{k, \mu} \left(-\frac{1}{\tau} \right) = \frac{(-1)^\iota (-i)^{1/2}}{\sqrt{M}} \frac{\tau^{\kappa+\iota+\frac{1}{2}}}{(2\pi i)^\kappa} \sum_{\nu \in L'/L} e^{2\pi i \mu \nu} \sum_{r=0}^\kappa d_{\kappa, \iota,r} \tau^r \theta_{2r+\iota, \nu}(\tau),
	\]
	where we put
	\[
		d_{\kappa,\iota,r} = \begin{cases}
			\dfrac{(4\pi i)^r (2\kappa+\iota)!}{2^\kappa (\kappa-r)! (2r+\iota)!} &\text{if } 0 \leq r \leq \kappa,\\
			0 &\text{if otherwise}.
		\end{cases}
	\]
\end{prop}

\begin{proof}
    It follows from induction on $\kappa$. The initial case of $\kappa=0$ holds by \cref{theta-01-S-trans}. If the $S$-transformation for $k = 2\kappa+\iota$ holds, then by taking the derivative in $\tau$ on both sides, we have
    \begin{align*}
    	\theta_{k+2, \mu} \left(-\frac{1}{\tau}\right) &= \frac{(-1)^\iota (-i)^{1/2}}{\sqrt{M}} \frac{\tau^{(\kappa+1)+\iota+\frac{1}{2}}}{(2\pi i)^{\kappa+1}} \sum_{\nu \in L'/L} e^{2\pi i \mu \nu}\\
    	&\quad \times \sum_{r=0}^{\kappa+1} 2\bigg((\kappa+\iota+r+1/2) d_{\kappa, \iota,r} + \pi i d_{\kappa, \iota,r-1} \bigg) \tau^r \theta_{2r+\iota, \nu}(\tau).
    \end{align*}
    The recursion
    \begin{align}\label{d-recursion}
        d_{\kappa+1,\iota,r} = 2\bigg((\kappa+\iota+r+1/2) d_{\kappa, \iota,r} + \pi i d_{\kappa, \iota,r-1} \bigg)
    \end{align}
    yields the claim for $k+2 = 2(\kappa+1)+\iota$.
\end{proof}

\subsection{Bivariate false theta functions}
\label{sec:Sfth-1}

We now recall Bringmann-Nazaroglu's function introduced in~\cite{BringmannNazaroglu2019}. For $(\tau,w,x) \in \bbH \times \bbH \times \mathbb{R}$ and $k \in \mathbb{Z}_{\geq 0}$, we define the function $F_{k,\tau,w}(x)$ by
\[
	F_{k,\tau, w}(x) = \sqrt{i(w-\tau)} x^k \erf \left(-i \sqrt{\pi i (w-\tau)} x \right) e^{\pi ix^2 \tau},
\]
where $\erf(z) = \frac{2}{\sqrt{\pi}} \int_0^z e^{-t^2} dt$ denotes the \emph{error function}.
Then the authors showed in~\cite[Lemma 2.1]{BringmannNazaroglu2019} the formula of its Fourier transform for $k = 0$.

\begin{lem}
    \begin{align}\label{eq:Fourier-F}
        \mathcal{F}(F_{0, \tau,w})(x) = i^{1/2} w^{1/2} F_{0, -\frac{1}{\tau}, -\frac{1}{w}}(x).
    \end{align}
\end{lem}

By using this lemma and Poisson's summation formula, the authors derived the $S$-transformation~\eqref{eq:BN-Strans} of the false theta function for $k=0$. (We will explain the relationship between the function $F_{0,\tau,w}(x)$ and the false theta function $\widetilde{\theta}_{0,\mu}(\tau)$ in \cref{theta-hat-limit}). 
To generalize the $S$-transformation to the case of $k>0$, we need to calculate the Fourier transform of $F_{1,\tau,w}(x)$ as the first step. 

\begin{lem}
	\[
		\mathcal{F}(F_{1,\tau,w})(x) = i^{1/2} w^{1/2} \left(\tau^{-1} F_{1,-\frac{1}{\tau}, -\frac{1}{w}}(x) - \frac{1}{\pi i} \frac{w - \tau}{\tau w} e^{-\pi i x^2/w} \right).
	\]
\end{lem}

\begin{proof}
	It follows from taking the derivative of \eqref{eq:Fourier-F} in $x$.
\end{proof}

Based on Bringmann-Nazaroglu's idea in~\cite{BringmannNazaroglu2019}, we introduce two types of theta functions with two variables $(\tau,w) \in \bbH \times \bbH$,
\begin{align}\label{two-theta-series}
\begin{split}
	\widehat{\Theta}_{k,\mu}(\tau, w) &= \sum_{n \in L + \mu} n^k \erf \left(-i \sqrt{\pi i(w-\tau)} n \right) e^{\pi i n^2 \tau}, \text{ and}\\
	\widetilde{\Theta}_{k, \mu}(\tau, w) &= \sum_{n \in L + \mu} F_{k,\tau, w}(n) = \sqrt{i(w-\tau)} \widehat{\Theta}_{k,\mu}(\tau, w).
\end{split}
\end{align}
First, we have to check the convergence of the sums. By the definition of the error function, we have
\begin{align}\label{eq:erf-integral}
   \erf \left(-i \sqrt{\pi i(w-\tau)} n \right) = -2i \sqrt{i(w-\tau)} n \int_0^1 e^{\pi i(w-\tau) n^2 t^2} dt.
\end{align}
The summand is bounded as
\begin{align*}
    \left\lvert n^k \erf \left(-i \sqrt{\pi i(w-\tau)} n \right) e^{\pi i n^2 \tau} \right\rvert \leq 2 \lvert w-\tau \rvert^{1/2} n^{k+1} e^{-\pi n^2 \min(\Im(\tau), \Im(w))}.
\end{align*}
Thus the series in \eqref{two-theta-series} are absolutely and uniformly convergent on compact subsets of $\bbH \times \bbH$.

Next, we explain the relationship between the theta functions in \eqref{two-theta-series} and false theta functions. For a fixed $\tau \in \bbH$, the branch cut defined by $\sqrt{i(w-\tau)}$ is along the half-line $\{w=\tau+iy \in \mathbb{C} \mid y>0\}$. By the fact that $\lim_{x \to \pm \infty}\erf(x)=\pm 1$, we have the following.

\begin{lem}\label{theta-hat-limit}
	For any $k \geq 0$ and $\epsilon > 0$,
	\[
		\lim_{t \to \infty} \widehat{\Theta}_{k, \mu}(\tau, \tau + it + \epsilon) = \widetilde{\theta}_{k, \mu}(\tau).
	\]
\end{lem}

In this sense, the theta function $\widehat{\Theta}_{k,\mu}(\tau,w)$ is a two-variable generalization of false theta functions. For $k=0, 1$, by applying Poisson's summation formula and the Fourier transform for $F_{k,\tau,w}(x)$ to the theta function $\widetilde{\Theta}_{k,\mu}(\tau,w)$, we obtain the following $S$-transformation formulas.

\begin{prop}\label{Theta-tilde-S-Trans}
	For $k=0,1$, we have
	\begin{align*}
		\widetilde{\Theta}_{0,\mu} \left(-\frac{1}{\tau}, -\frac{1}{w} \right) &= \frac{-(-i)^{1/2} w^{-1/2}}{\sqrt{M}} \sum_{\nu \in L'/L} e^{2\pi i \mu \nu} \widetilde{\Theta}_{0, \nu}(\tau, w)
	\end{align*}
	and
	\begin{align*}
		\widetilde{\Theta}_{1, \mu} \left(-\frac{1}{\tau}, -\frac{1}{w} \right) &= \frac{(-i)^{1/2} w^{-1/2}}{\sqrt{M}} \sum_{\nu \in L'/L} e^{2\pi i \mu \nu} \left(\tau \widetilde{\Theta}_{1,\nu}(\tau, w) + \frac{w-\tau}{\pi i} \theta_{0,\nu}(w) \right).
	\end{align*}
\end{prop}

For $k=0$, taking the limit $w \to i\infty$ along $w = \tau+it+\epsilon$ ($\epsilon > 0$) gives the $S$-transformation of $\widetilde{\theta}_{0,\mu}(\tau)$ as explained in the next subsection. On the other hand, for $k=1$, the term $w - \tau$ diverges in a similar limit. In order to establish a suitable $S$-transformation formula one should take care of possible divergences.

For the convenience of later use, let us show an integral expression of the false theta function.

\begin{lem}\label{integrate-branch}
	We assume that $\Re(\tau) \neq \Re(w)$. For $k \geq 0$, we have
	\[
		\widehat{\Theta}_{k,\mu}(\tau,w) = -i \sgn(\Re(w-\tau)) \int_\tau^w \frac{\theta_{k+1, \mu}(z)}{\sqrt{-i(z-\tau)}} dz,
	\]
	where the integration path avoids the branch cut defined by $\sqrt{-i(z-\tau)}$, that is, $\{z = \tau-iy \in \mathbb{C} \mid y >0\}$.
\end{lem}

\begin{proof}
    By \eqref{eq:erf-integral}, we have
    \begin{align*}
		\widehat{\Theta}_{k,\mu}(\tau, w) = -2i \sqrt{i(w-\tau)} \int_0^1 \sum_{n \in L + \mu} n^{k+1} e^{\pi i (w-\tau)n^2t^2 + \pi i n^2 \tau} dt.
	\end{align*}
	Changing variable via $(w-\tau)t^2 + \tau = z$ yields the integral expression
	\begin{align}\label{eq:hatTheta-int}
		\widehat{\Theta}_{k,\mu}(\tau,w) &= \int_\tau^w \frac{\theta_{k+1, \mu}(z)}{\sqrt{i(z-\tau)}} dz = -i \sgn(\Re(w-\tau)) \int_\tau^w \frac{\theta_{k+1, \mu}(z)}{\sqrt{-i(z-\tau)}} dz.
	\end{align}
	Here we use that $\sqrt{i(z-\tau)} = i \sgn(\Re(w-\tau)) \sqrt{-i(z-\tau)}$ holds for $z \in \bbH$ on the line segment connecting $\tau$ and $w$.
\end{proof}

From \cref{theta-hat-limit}, we also have the expression
\begin{align}\label{false-theta-integral-exp}
	\widetilde{\theta}_{k, \mu}(\tau) = -i \int_\tau^{i\infty} \frac{\theta_{k+1,\mu}(z)}{\sqrt{-i(z-\tau)}} dz,
\end{align}
where we note that $\theta_{k+1, \mu}(z)$ decays exponentially as $z \to i\infty$.

\begin{rem}
    Although changing the sign of $i$ in the denominator in \eqref{eq:hatTheta-int} seems to be a minor change, the expression given in \eqref{false-theta-integral-exp} might be better because we do not need to care about the sign of $\epsilon$ in the path of integration.
\end{rem}

\subsection{False theta functions}
\label{sec:false-S-trans}

As a corollary of \cref{Theta-tilde-S-Trans}, we show how to derive the $S$-transformation of $\widetilde{\theta}_{0,\mu}(\tau)$ given in \eqref{eq:BN-Strans}.

\begin{lem}\label{false-theta-k0}
	We assume that $\Re(\tau) \neq 0$. The false theta function $\widetilde{\theta}_{0,\mu}(\tau)$ satisfies the following $S$-transformation.
	\begin{align*}
		&\widetilde{\theta}_{0, \mu} \left(-\frac{1}{\tau} \right) + \sgn(\Re(\tau)) \frac{(-i\tau)^{1/2}}{\sqrt{M}} \sum_{\nu \in L'/L} e^{2\pi i \mu \nu} \widetilde{\theta}_{0,\nu}(\tau)\\
		&=  - i \int_0^{i\infty} \frac{\theta_{1,\mu}(z)}{\sqrt{-i(z+1/\tau)}} dz.
	\end{align*}
\end{lem}

\begin{proof}
	For the first equation in \cref{Theta-tilde-S-Trans}, by taking the limit as $w \to i\infty$ along $w = \tau + it + \epsilon$ ($\epsilon > 0$), we have
	\begin{align}\label{Theta-hat-0}
		\lim_{t \to \infty} \widehat{\Theta}_{0, \mu} \left(-\frac{1}{\tau}, -\frac{1}{\tau+it+\epsilon} \right) = \frac{-(-i\tau)^{1/2}}{\sqrt{M}} \sum_{\nu \in L'/L} e^{2\pi i \mu \nu} \widetilde{\theta}_{0,\nu}(\tau),
	\end{align}
	where we recall the relation $\widetilde{\Theta}_{0, \mu}(\tau,w) = \sqrt{i(w-\tau)} \widehat{\Theta}_{0,\mu}(\tau, w)$.
	Since $\theta_{1,\mu}(z)$ decays exponentially as $z \to 0$ by \cref{theta-01-S-trans}, the limit converges to
	\[
		\widehat{\Theta}_{0, \mu} \left(-\frac{1}{\tau}, 0 \right) = i \sgn(\Re(\tau)) \int_0^{-\frac{1}{\tau}} \frac{\theta_{1,\mu}(z)}{\sqrt{-i(z+1/\tau)}} dz
	\]
	by \cref{integrate-branch}. By adding
	\begin{align*}
		i \sgn(\Re(\tau)) \int_{-\frac{1}{\tau}}^{i\infty} \frac{\theta_{1,\mu}(z)}{\sqrt{-i(z+1/\tau)}} dz = -\sgn(\Re(\tau)) \widetilde{\theta}_{0, \mu} \left(-\frac{1}{\tau} \right)
	\end{align*}
	to both sides of \eqref{Theta-hat-0}, we obtain
	\begin{align*}
		&i \sgn(\Re(\tau)) \int_0^{i\infty} \frac{\theta_{1,\mu}(z)}{\sqrt{-i(z+1/\tau)}} dz\\
		&= \frac{-(-i\tau)^{1/2}}{\sqrt{M}} \sum_{\nu \in L'/L} e^{2\pi i \mu \nu} \widetilde{\theta}_{0,\nu}(\tau) -\sgn(\Re(\tau)) \widetilde{\theta}_{0, \mu} \left(-\frac{1}{\tau} \right),
	\end{align*}
	which concludes the proof.
\end{proof}

Since
\[
	\frac{d}{d\tau} \widetilde{\theta}_{k, \mu}(\tau) = \pi i \widetilde{\theta}_{k+2,\mu}(\tau)
\]
holds, it is essentially enough to obtain the $S$-transformations for general $\widetilde{\theta}_{k,\mu}(\tau)$ that we establish that for $k=0$ and $k = 1$. While the case of $k=1$ requires careful treatment of convergence in the limit for $w$, we will simply avoid this problem by considering a suitable linear combination of false theta functions. 

We let $k \geq 0$ and $d \geq 2$ be integers. For any $\bm{\mu} = (\mu_j)_{1 \leq j \leq d} \in {L'}^d$ and $\bm{a} = (a_j)_{1 \leq j \leq d} \in \mathbb{C}^d$ satisfying $\sum_{j=1}^d a_j = 0$, we put
\begin{align}\label{mu-a}
	\widetilde{\theta}_{k, \bm{\mu}, \bm{a}}(\tau) = \sum_{j=1}^d a_j \widetilde{\theta}_{k, \mu_j}(\tau), \quad \theta_{k, \bm{\mu}, \bm{a}} (\tau) = \sum_{j=1}^d a_j \theta_{k, \mu_j}(\tau).
\end{align}
A simple observation reveals that $\widetilde{\theta}_{k, -\mu}(\tau) = (-1)^{k+1} \widetilde{\theta}_{k, \mu}(\tau)$ and $\theta_{k, -\mu}(\tau) = (-1)^k \theta_{k, \mu}(\tau)$ hold. Thus, if $k$ is even, a single false theta function $\widetilde{\theta}_{k, \mu}(\tau)$ is expressed by the above expression as $2\widetilde{\theta}_{k,\mu}(\tau) = \widetilde{\theta}_{k, \mu}(\tau) - \widetilde{\theta}_{k, -\mu}(\tau)$.
The linear combination satisfies the following $S$-transformation.

\begin{lem}\label{false-theta-k1}
	The notations are the same as above and we assume that $\Re(\tau) \neq 0$. Then, we have
	\begin{align*}
		&\widetilde{\theta}_{1, \bm{\mu}, \bm{a}} \left(-\frac{1}{\tau}\right) + \sgn(\Re(\tau)) \frac{-i (-i\tau)^{3/2}}{\sqrt{M}} \sum_{\nu \in L'/L} \left(\sum_{j=1}^d a_j e^{2\pi i \mu_j \nu} \right) \widetilde{\theta}_{1,\nu}(\tau)\\
		&=-i \int_0^{i\infty} \frac{\theta_{2, \bm{\mu}, \bm{a}}(z)}{\sqrt{-i(z+1/\tau)}} dz.
	\end{align*}
\end{lem}

\begin{proof}
	By \cref{Theta-tilde-S-Trans} and \cref{integrate-branch}, we have
	\begin{align*}
		&-i \sgn \left(\Re \left(-\frac{1}{w} + \frac{1}{\tau} \right) \right) \int_{-\frac{1}{\tau}}^{-\frac{1}{w}} \frac{\theta_{2, \bm{\mu}, \bm{a}}(z)}{\sqrt{-i(z+1/\tau)}} dz\\
		&= \frac{(-i)^{1/2}}{\sqrt{M}} \frac{w^{-1/2}}{\sqrt{i(-1/w+1/\tau)}}\\
		&\qquad \times \sum_{\nu \in L'/L} \left(\sum_{j=1}^d a_j e^{2\pi i \mu_j \nu} \right) \left(\tau \sqrt{i(w-\tau)} \widehat{\Theta}_{1,\nu}(\tau, w) + \frac{w-\tau}{\pi i} \theta_{0,\nu}(w) \right).
	\end{align*}
	The left-hand side converges as $w \to i\infty$ because the theta function $\theta_{2, \bm{\mu}, \bm{a}}(z)$ exponentially decays as $z \to 0$ by \cref{theta-S-transform}. More precisely, the theta function $\theta_{2,\mu}(z)$ satisfies
	\begin{align*}
		\theta_{2,\mu} \left(-\frac{1}{z}\right) = \frac{(-i)^{1/2} z^{3/2}}{2\pi i \sqrt{M}} \sum_{\nu \in L'/L} e^{2\pi i \mu \nu} \bigg(\theta_{0,\nu}(z) + 2\pi i z \theta_{2,\nu}(z) \bigg).
	\end{align*}
	Since the term $z^{3/2} \theta_{0,0}(z)$ on the right-hand side has a polynomial growth as $z \to i\infty$, the single $\theta_{2,\mu}(z)$ diverges as $z \to 0$. However, this growth is cancelled out by taking the linear combination of theta functions.
	
	Similarly, the right-hand side also converges. In fact, $\theta_{0, \nu}(w)$ for $\nu \neq 0$ decays exponentially as $w \to i\infty$ and $\theta_{0,0}(w)$ vanishes by the term $\sum_{j=1}^d a_j e^{2\pi i \mu_j \nu}$ for $\nu = 0$. Therefore, we obtain
	\begin{align*}
		\sgn(\Re(\tau)) \int_0^{-\frac{1}{\tau}} \frac{\theta_{2, \bm{\mu}, \bm{a}}(z)}{\sqrt{-i(z+1/\tau)}} dz = \frac{(-i\tau)^{3/2}}{\sqrt{M}} \sum_{\nu \in L'/L} \left(\sum_{j=1}^d a_j e^{2\pi i \mu_j \nu} \right) \widetilde{\theta}_{1,\nu}(\tau).
	\end{align*}
	Since each function $\theta_{2,\mu_j}(z)$ has an exponential decay at infinity, we can add 
	\begin{align*}
		\sgn(\Re(\tau)) \int_{-\frac{1}{\tau}}^{i\infty} \frac{\theta_{2, \bm{\mu}, \bm{a}}(z)}{\sqrt{-i(z+1/\tau)}} dz = i \sgn(\Re(\tau)) \widetilde{\theta}_{1, \bm{\mu}, \bm{a}} \left(-\frac{1}{\tau}\right),
	\end{align*}
	to both sides. Thus we get
	\begin{align*}
		&\sgn(\Re(\tau)) \int_0^{i\infty} \frac{\theta_{2, \bm{\mu}, \bm{a}}(z)}{\sqrt{-i(z+1/\tau)}} dz\\
		&= \frac{(-i\tau)^{3/2}}{\sqrt{M}} \sum_{\nu \in L'/L} \left(\sum_{j=1}^d a_j e^{2\pi i \mu_j \nu} \right) \widetilde{\theta}_{1,\nu}(\tau) + i \sgn(\Re(\tau)) \widetilde{\theta}_{1, \bm{\mu}, \bm{a}} \left(-\frac{1}{\tau}\right),
	\end{align*}
	which gives the desired equation.
\end{proof}

Combining \cref{false-theta-k0} and \cref{false-theta-k1}, for $k=0,1$,
\begin{align}\label{false-theta-01}
	&\widetilde{\theta}_{k, \bm{\mu}, \bm{a}} \left(-\frac{1}{\tau} \right) + \sgn(\Re(\tau)) \frac{(-1)^k (-i)^{1/2} \tau^{k+\frac{1}{2}}}{\sqrt{M}} \sum_{\nu \in L'/L} \left(\sum_{j=1}^d a_j e^{2\pi i \mu_j \nu} \right) \widetilde{\theta}_{k,\nu}(\tau) \nonumber\\
		&= - i \int_0^{i\infty} \frac{\theta_{k+1,\bm{\mu}, \bm{a}}(z)}{\sqrt{-i(z+1/\tau)}} dz.
\end{align}
By taking the derivatives in $\tau$ repeatedly, we obtain the following general $S$-transformation formulas.

\begin{prop}\label{false-theta-general-k}
	The notations are the same as before.
	For $\Re(\tau) \neq 0$ and $k = 2\kappa + \iota$ with $\kappa \in \mathbb{Z}_{\geq 0}$ and $\iota \in \{0,1\}$, we have
	\begin{align*}
		&\widetilde{\theta}_{k, \bm{\mu}, \bm{a}} \left(-\frac{1}{\tau} \right)\\
		&+ \sgn(\Re(\tau)) \frac{(-1)^\iota (-i)^{1/2} \tau^{\kappa+\iota+\frac{1}{2}}}{(2\pi i)^\kappa \sqrt{M}} \sum_{\nu \in L'/L} \left(\sum_{j=1}^d a_j e^{2\pi i \mu_j \nu} \right)\sum_{r=0}^\kappa d_{\kappa,\iota,r} \tau^r \widetilde{\theta}_{2r+\iota,\nu}(\tau)\\
		&= - i \int_0^{i\infty} \frac{\theta_{k+1, \bm{\mu}, \bm{a}}(z)}{\sqrt{-i(z+1/\tau)}} dz,
	\end{align*}
	where $d_{\kappa,\iota,r}$ is given in \cref{theta-S-transform}.
\end{prop}

\begin{proof}
	The idea is almost the same as that for \cref{theta-S-transform}. The initial case of $\kappa=0$ holds by \eqref{false-theta-01}. We assume that the desired $S$-transformation for $k=2\kappa+\iota$ holds. The derivative in $\tau$ of the left-hand side is given by
	\begin{align*}
		&\frac{\pi i}{\tau^2} \bigg(\widetilde{\theta}_{k+2, \bm{\mu}, \bm{a}} \left(-\frac{1}{\tau} \right) + \sgn(\Re(\tau)) \frac{(-1)^\iota (-i)^{1/2} \tau^{(\kappa+1)+\iota+\frac{1}{2}}}{(2\pi i)^{\kappa+1} \sqrt{M}} \\
		& \times \sum_{\nu \in L'/L} \left(\sum_{j=1}^d a_j e^{2\pi i \mu_j \nu} \right) \sum_{r=0}^{\kappa+1} 2 \bigg( (\kappa+\iota+r+1/2) d_{\kappa,\iota,r} + \pi i d_{\kappa,\iota,r-1} \bigg) \tau^r \widetilde{\theta}_{2r+\iota,\nu}(\tau) \bigg).
	\end{align*}
	On the other hand, the derivative of the right-hand side equals
	\begin{align*}
		- \frac{1}{2\tau^2} \int_0^{i\infty} \frac{\theta_{k+1,\bm{\mu}, \bm{a}}(z)}{(-i(z+1/\tau))^{3/2}} dz = \frac{\pi}{\tau^2} \int_0^{i\infty} \frac{\theta_{k+3, \bm{\mu}, \bm{a}}(z)}{\sqrt{-i(z+1/\tau)}} dz
	\end{align*}
	by the partial integration. Here we recall that the integrand $\theta_{k+1,\bm{\mu}, \bm{a}}(z)$ decays exponentially as $z \to 0$ and $z \to i\infty$ by \cref{theta-S-transform}. Finally, the recursion \eqref{d-recursion} yields the claim for $k+2 = 2(\kappa+1) + \iota$.
\end{proof}

\section{Decomposition of the homological block}
\label{sec:de-ftf}

In the rest of the article, we always assume $n \geq 3$. In this section, we deform the homological block $\Phi(q)$ defined in \cref{defi:WRT-function} into a sum of false theta functions. First, we focus on a part of the homological block defined by
\begin{align*}
	\Psi(q) := \sum_{(\varepsilon_1, \dots, \varepsilon_n) \in \{\pm 1\}^n} \varepsilon_1 \cdots \varepsilon_n \sum_{m=0}^\infty {m+n-3 \choose n-3} q^{\frac{1}{4P} \left(2Pm+P \left(n-2 + \sum_{j=1}^n \frac{\varepsilon_j}{p_j}\right) \right)^2}.
\end{align*}
Since $p_j$'s are pairwise coprime, the number $\sum_{j=1}^n \varepsilon_j/p_j$ can never be an integer. For $\bm{p} = (p_1, \dots, p_n)$ and $\bm{\varepsilon} = (\varepsilon_1, \dots, \varepsilon_n)$, there uniquely exist $m_0 = m_0(\bm{p}, \bm{\varepsilon}) \in \mathbb{Z}$ and $\ell = \ell(\bm{p}, \bm{\varepsilon}) \in \{1,2,\dots,2P-1\}$ such that
\begin{align}\label{m0-l-def}
	P \left(n - 2 + \sum_{j=1}^n \frac{\varepsilon_j}{p_j} \right) = 2P m_0 + \ell.
\end{align}
We note that a direct calculation yields $m_0(\bm{p}, -\bm{\varepsilon}) = n-3-m_0(\bm{p}, \bm{\varepsilon})$ and $\ell(\bm{p}, -\bm{\varepsilon}) = 2P - \ell(\bm{p},\bm{\varepsilon})$ for $-\bm{\varepsilon} = (-\varepsilon_1, \dots, -\varepsilon_n)$. With these notations, we show two expressions of $\Psi(q)$ in \cref{Psi-expand-1} and \cref{Psi-expand-2}.

\begin{lem}\label{Psi-expand-1}
	\begin{align*}
		\Psi(q) &= \sum_{(\varepsilon_1, \dots, \varepsilon_n) \in \{\pm 1\}^n} \frac{\varepsilon_1 \cdots \varepsilon_n}{(n-3)!} \sum_{k=0}^{n-3} \left(\sum_{l=k}^{n-3} \frac{(-\ell-2Pm_0)^{l-k}}{(2P)^l} \st{n-2}{l+1} {l \choose k} \right) \sum_{\substack{m \equiv \ell\ (2P) \\ m \geq 2Pm_0+\ell}} m^k q^{\frac{m^2}{4P}},
	\end{align*}
	where $\st{\cdot}{\cdot}$ is the Stirling number of the first kind defined by $\st{0}{0} = 1, \st{n}{0} = \st{0}{k} = 0$, and the recursion $\st{n+1}{k} = \st{n}{k-1} + n \st{n}{k}$.
\end{lem}

\begin{proof}
	By the definitions of $m_0$ and $\ell$ in \eqref{m0-l-def},
	\begin{align}\label{Psi-sum-sum}
		\Psi(q) &= \sum_{(\varepsilon_1, \dots, \varepsilon_n) \in \{\pm 1\}^n} \varepsilon_1 \cdots \varepsilon_n \sum_{m=0}^\infty {m+n-3 \choose n-3} q^{\frac{(2P(m+ m_0) + \ell)^2}{4P}}\\
			&= \sum_{(\varepsilon_1, \dots, \varepsilon_n) \in \{\pm 1\}^n} \varepsilon_1 \cdots \varepsilon_n \sum_{m=m_0}^\infty {m-m_0+n-3 \choose n-3} q^{\frac{(2Pm + \ell)^2}{4P}}. \nonumber
	\end{align}
	By the well-known fact ${x+n \choose n} = \frac{1}{n!} \sum_{l=0}^n \st{n+1}{l+1} x^l$,
	the inner sum becomes
	\begin{align*}
		&\sum_{m=m_0}^\infty {m-m_0+n-3 \choose n-3} q^{\frac{(2Pm+\ell)^2}{4P}}\\
			&\quad = \frac{1}{(n-3)!} \sum_{m=m_0}^\infty \sum_{l=0}^{n-3} \st{n-2}{l+1} \left(m+\frac{\ell}{2P} - \frac{\ell}{2P} -m_0 \right)^l q^{\frac{(2Pm+\ell)^2}{4P}}\\
			&\quad = \frac{1}{(n-3)!} \sum_{l=0}^{n-3} \frac{1}{(2P)^l} \st{n-2}{l+1} \sum_{k=0}^l {l \choose k} (-\ell-2Pm_0)^{l-k} \sum_{\substack{m \equiv \ell\ (2P) \\ m \geq 2Pm_0+\ell}} m^k q^{\frac{m^2}{4P}},
	\end{align*}
	which concludes the proof.
\end{proof}

On the other hand, by replacing $\bm{\varepsilon} \mapsto -\bm{\varepsilon}$ in the sum in \eqref{Psi-sum-sum}, we obtain another expression
\begin{align*}
	\Psi(q) &= \sum_{(\varepsilon_1, \dots, \varepsilon_n) \in \{\pm 1\}^n} (-\varepsilon_1) \cdots (-\varepsilon_n) \\
	    &\quad \times \sum_{m=0}^\infty (-1)^{n-3} {-m-1 \choose n-3} q^{\frac{(2P(m+ m_0(\bm{p}, -\bm{\varepsilon})) + \ell(\bm{p}, -\bm{\varepsilon}))^2}{4P}}\\
		&= - \sum_{(\varepsilon_1, \dots, \varepsilon_n) \in \{\pm 1\}^n} \varepsilon_1 \cdots \varepsilon_n \sum_{m \leq 0} {m-1 \choose n-3} q^{\frac{(2P(m- n+2+m_0(\bm{p}, \bm{\varepsilon})) +\ell(\bm{p},\bm{\varepsilon}))^2}{4P}},
\end{align*}
where we use the equation ${-x-1 \choose n} = (-1)^n {x+n \choose n}$. By the same calculation as in the above proof, we have another expression.

\begin{lem}\label{Psi-expand-2}
\begin{align*}
	\Psi(q) &= -\sum_{(\varepsilon_1, \dots, \varepsilon_n) \in \{\pm 1\}^n} \frac{\varepsilon_1 \cdots \varepsilon_n}{(n-3)!}\\
		&\quad \times \sum_{k=0}^{n-3} \left(\sum_{l=k}^{n-3} \frac{(-\ell-2Pm_0)^{l-k}}{(2P)^l} \st{n-2}{l+1} {l \choose k} \right) \sum_{\substack{m \equiv \ell\ (2P) \\ m \leq 2P(m_0-n+2)+\ell}} m^k q^{\frac{m^2}{4P}}.
\end{align*}
\end{lem}

Combining \cref{Psi-expand-1} and \cref{Psi-expand-2}, we have the following.

\begin{prop}\label{Psi-decomp}
	Let $m_0$ and $\ell$ be as in \eqref{m0-l-def}. Then we have
	\begin{align}\label{eq:Psi-decomp}
		\Psi(q) = \frac{1}{2} \sum_{(\varepsilon_1, \dots, \varepsilon_n) \in \{\pm 1\}^n} \varepsilon_1 \cdots \varepsilon_n \sum_{k=0}^{n-3} c_{\bm{\varepsilon}}(k) \sum_{m \equiv \ell\ (2P)} \sgn_{\bm{\varepsilon}}(m) m^k q^{\frac{m^2}{4P}},
	\end{align}
	where we put
	\[
		c_{\bm{\varepsilon}}(k) = \frac{(-1)^k}{P^k (n-3)!} \sum_{l=k}^{n-3} (-2)^{-l} \left(n-2+\sum_{j=1}^n \frac{\varepsilon_j}{p_j} \right)^{l-k}\st{n-2}{l+1} {l \choose k} \in \mathbb{Q}
	\]
	and
	\[
		\sgn_{\bm{\varepsilon}} (m) = \begin{cases}
			+1 &\text{if } m \geq 2Pm_0 + \ell,\\
			-1 &\text{if } m \leq 2P(m_0-n+2) + \ell,\\
			0 &\text{if otherwise}.
		\end{cases}
	\]
\end{prop}

If the term $\sgn_{\bm{\varepsilon}} (m)$ in \eqref{eq:Psi-decomp} is replaced by the usual sign function, the inner sum becomes a false theta function. To finish this section, we define a polynomial $\mathcal{P}(q) \in \mathbb{Q}[q^{\frac{1}{4P}}]$ by
\begin{align}\label{def-P(q)}
	\mathcal{P}(q) = \frac{1}{2} \sum_{(\varepsilon_1, \dots, \varepsilon_n) \in \{\pm 1\}^n} \varepsilon_1 \cdots \varepsilon_n \sum_{k=0}^{n-3} c_{\bm{\varepsilon}}(k) \sum_{m \equiv \ell\ (2P)} (\sgn_{\bm{\varepsilon}}(m) - \sgn(m)) m^k q^{\frac{m^2}{4P}}.
\end{align}
Therefore we obtain the decomposition of the homological block.
\begin{thm}\label{Psi-P=theta-tilde}
	Let $\mathcal{P}(q)$ and $c_{\bm{\varepsilon}}(k)$ be as above. Then we have
	\[
		\Psi(q) - \mathcal{P}(q) = \frac{1}{2} \sum_{(\varepsilon_1, \dots, \varepsilon_n) \in \{\pm 1\}^n} \varepsilon_1 \cdots \varepsilon_n \sum_{k=0}^{n-3} (2P)^{k/2} c_{\bm{\varepsilon}}(k) \widetilde{\theta}_{k, \frac{\ell}{\sqrt{2P}}}(\tau),
	\]
	where the lattice is set by $L = \sqrt{2P}\mathbb{Z}$, and recall that $\Psi(q)$ is a factor of the homological block $\Phi(q)$, namely,
	\[
		\Phi(q) = \frac{(-1)^n}{2(q^{\frac{1}{2}} - q^{-\frac{1}{2}})} q^{-\frac{1}{4} \Theta_0} \Psi(q).
	\]
\end{thm}

\begin{proof}
    By \cref{Psi-decomp} and the definition of the polynomial $\mathcal{P}(q)$ in~\eqref{def-P(q)}, we have
	\[
		\Psi(q) - \mathcal{P}(q) = \frac{1}{2} \sum_{(\varepsilon_1, \dots, \varepsilon_n) \in \{\pm 1\}^n} \varepsilon_1 \cdots \varepsilon_n \sum_{k=0}^{n-3} c_{\bm{\varepsilon}}(k) \sum_{m \equiv \ell\ (2P)} \sgn(m) m^k q^{\frac{m^2}{4P}}.
	\]
	For a lattice $L = \sqrt{2P}\mathbb{Z}$ and $\mu = \ell/\sqrt{2P} \in L'/L$ ($\ell = 0, 1, \dots, 2P-1$), the false theta function has an expression
	\[
		\widetilde{\theta}_{k,\frac{\ell}{\sqrt{2P}}}(\tau) = (2P)^{-k/2} \sum_{m \equiv \ell\ (2P)} \sgn(m) m^k q^{\frac{m^2}{4P}},
	\]
	which finishes the proof.
\end{proof}

\begin{ex}[The case of $(p_1, p_2, p_3) = (2,3,5)$]\label{2-3-5}
	For every $\varepsilon = (\varepsilon_1,\varepsilon_2,\varepsilon_3)$, the numbers $m_0$ and $\ell$ are given as the following table.
	Then we have $\mathcal{P}(q) = - 2 q^{\frac{1}{120}}$ and
\[
	\Psi(q) + 2 q^{\frac{1}{120}} = \frac{1}{2} \sum_{(\varepsilon_1, \varepsilon_2, \varepsilon_3) \in \{\pm 1\}^3} \varepsilon_1 \varepsilon_2 \varepsilon_3 \sum_{m \equiv \ell\ (60)} \sgn(m) q^{\frac{m^2}{120}},
\]
which equals $\widetilde{\Theta}_+(z)$ defined in Lawrence-Zagier~\cite{LawrenceZagier1999}.

\begin{table}[h]
    \begin{center}
	\begin{tabular}{|c|c|c|c|c||c|c|c|c|c|}  \hline
		$\varepsilon_1$ & $\varepsilon_2$ & $\varepsilon_3$ & $m_0$ & $\ell$ & $\varepsilon_1$ & $\varepsilon_2$ & $\varepsilon_3$ & $m_0$ & $\ell$ \\ \hline \hline
		$+1$ & $+1$ & $+1$ & $1$ & $1$ & $+1$ & $-1$ & $-1$ & $0$ & $29$\\ \hline
		$+1$ & $+1$ & $-1$ & $0$ & $49$ & $-1$ & $+1$ & $-1$ & $0$ & $19$\\ \hline
		$+1$ & $-1$ & $+1$ & $0$ & $41$ & $-1$ & $-1$ & $+1$ & $0$ & $11$\\ \hline
		$-1$ & $+1$ & $+1$ & $0$ & $31$ & $-1$ & $-1$ & $-1$ & $-1$ & $59$\\ \hline
	\end{tabular}
	\end{center}
	\caption{The list of $(m_0, \ell)$ corresponding to each $\bm{\varepsilon} \in \{\pm 1\}^3$.}
\end{table}
\end{ex}

\section{Main theorems}
\label{sec:mt}

In the last section, we show the $S$-transformation of the function 
\begin{align}\label{Psi-hat-def}
	\widehat{\Psi}(\tau) := \Psi(q) - \mathcal{P}(q)
\end{align}
in \cref{thm:S-trans-WRT}. We recall that the function $\Psi(q)$ is a factor of the homological block $\Phi(q)$ as explained in \cref{Psi-P=theta-tilde}, and $\mathcal{P}(q)$ is a suitable polynomial defined in \eqref{def-P(q)}. As an application, we obtain explicit asymptotic expansion formulas (\cref{thm:asymp-Andersen} and \cref{CS}) for the WRT invariants $\tau_K$ for any Seifert fibered integral homology $3$-sphere. They give a new proof of a version by Andersen of the Witten asymptotic conjecture.

\subsection{Asymptotic expansions}
\label{sec:Lfunct}

We recall the asymptotic expansion of a certain series shown by Lawrence-Zagier~\cite[Section 3]{LawrenceZagier1999} and its generalization. Let $C: \mathbb{Z} \to \mathbb{C}$ be a periodic function whose period is $M$. Moreover, we assume that its mean value is $0$, that is,
\[
	\sum_{m=1}^M C(m) = 0.
\]
Then the Dirichlet series $L(s,C) = \sum_{m=1}^\infty C(m) m^{-s}$ defines a holomorphic function in $\Re(s) > 1$ and extends holomorphically to the whole $\mathbb{C}$. The special values at negative integers are given by
\begin{align}\label{eq:L-val}
	L(-r,C) = - \frac{M^r}{r+1} \sum_{m=1}^M C(m) B_{r+1} \left(\frac{m}{M} \right),
\end{align}
where $B_m(x)$ is the $m$-th Bernoulli polynomial defined by
\[
	\sum_{m=0}^\infty B_m(x) \frac{t^m}{m!} = \frac{t e^{xt}}{e^t-1}.
\]

In the same way as Lawrence-Zagier, or as Andersen-Misteg{\aa}rd~\cite[Section 4.1.2]{AndersenMistegard2018} showed, the following asymptotic expansion holds.

\begin{lem}\label{lem:asymp-L}
	For any integer $r \geq 0$,
	\[
		\sum_{m=1}^\infty m^r C(m) e^{-m^2 t} \sim \sum_{m=0}^\infty L(-2m-r, C) \frac{(-t)^m}{m!} \quad (t \to 0+).
	\]
\end{lem}

To show the $S$-transformation of $\widehat{\Psi}(\tau)$ and the asymptotic expansion of the WRT invariants, we consider the following special periodic function.

\begin{lem}
	Let $\ell = \ell(\bm{p}, \bm{\varepsilon})$ be as in \cref{sec:de-ftf}. For any integers $0 \leq k \leq n-3$ and $K \in \mathbb{Z}$, we define a periodic function $C_{k,K}: \mathbb{Z} \to \mathbb{C}$ whose period is $2P$ by
	\begin{align*}
		C_{k, K}(m) &= e^{2 \pi i \frac{-m^2 K}{4P}} \sum_{(\varepsilon_1, \dots, \varepsilon_n) \in \{\pm 1\}^n} \varepsilon_1 \cdots \varepsilon_n c_{\bm{\varepsilon}}(k) e^{2\pi i \frac{\ell m}{2P}}\\
			&= (-1)^{mn} e^{2\pi i \frac{-m^2 K}{4P}} \sum_{(\varepsilon_1, \dots, \varepsilon_n) \in \{\pm 1\}^n} \varepsilon_1 \cdots \varepsilon_n c_{\bm{\varepsilon}}(k) \prod_{j=1}^n e^{2\pi i \frac{\varepsilon_j m}{2p_j}},
	\end{align*}
	where $c_{\bm{\varepsilon}}(k) \in \mathbb{Q}$ is given in \cref{Psi-decomp}. Then the mean value of $C_{k,K}$ equals $0$, that is,
	\[
		\sum_{m=1}^{2P} C_{k, K}(m) = 0.
	\]
\end{lem}

\begin{proof}
	By the Chinese remainder theorem, for each $\bm{\varepsilon} = (\varepsilon_1, \dots, \varepsilon_n) \in \{\pm 1\}^n$, there exists a bijection $\phi_{\bm{\varepsilon}}: \mathbb{Z}/2P\mathbb{Z} \to \mathbb{Z}/2P\mathbb{Z}$ such that
	\[
		\begin{cases}
			\phi_{\bm{\varepsilon}}(m) \equiv \varepsilon_1 m \pmod{2p_1},\\
			\phi_{\bm{\varepsilon}}(m) \equiv \varepsilon_2 m \pmod{2p_2},\\
			\qquad \vdots\\
			\phi_{\bm{\varepsilon}}(m) \equiv \varepsilon_n m \pmod{2p_n}.
		\end{cases}
	\]
	By replacing $m \mapsto \phi_{\bm{\varepsilon}}(m)$,
	\begin{align*}
		\sum_{m=1}^{2P} C_{k,K}(m) = \sum_{(\varepsilon_1, \dots, \varepsilon_n) \in \{\pm 1\}^n} \varepsilon_1 \cdots \varepsilon_n c_{\bm{\varepsilon}}(k) \sum_{m=1}^{2P} (-1)^{mn} e^{2\pi i \frac{-m^2 K}{4P}} \prod_{j=1}^n e^{2\pi i \frac{m}{2p_j}}.
	\end{align*}
	Therefore, it suffices to show that
	\begin{align}\label{ep-sum-vanish}
		\sum_{(\varepsilon_1, \dots, \varepsilon_n) \in \{\pm 1\}^n} \varepsilon_1 \cdots \varepsilon_n c_{\bm{\varepsilon}}(k) = 0
	\end{align}
	for $0 \leq k \leq n-3$. By substituting the definition of $c_{\bm{\varepsilon}}(k)$, the left-hand side of \eqref{ep-sum-vanish} is expressed as 
	\begin{align}\label{eps-experss}
		\sum_{\substack{J = \{j_1, \dots, j_t\} \subset \{1, \dots, n\} \\ t \geq k+3}} e_k(J) \sum_{(\varepsilon_1, \dots, \varepsilon_n) \in \{\pm 1\}^n} \varepsilon_{j_1} \cdots \varepsilon_{j_t}
	\end{align}
	for some rational coefficients $e_k(J) \in \mathbb{Q}$. For each $J$, the inner sum equals $0$, that is, \eqref{ep-sum-vanish} holds.
\end{proof}

We make use of the following lemma in later calculations.

\begin{lem}\label{CK-sincosprod}
	For $0 \leq k \leq n-3$, we define $\widetilde{C}_k(m)$ satisfying
	\[
		C_{k,K}(m) = e^{2\pi i \frac{-m^2 K}{4P}} \widetilde{C}_{k}(m).
	\]
	If $m$ is divisible by at least $n-k-2$ of the $p_j$'s, then $\widetilde{C}_k (m) = 0$.
\end{lem}

\begin{proof}
	By \eqref{eps-experss}, we have
	\begin{align*}
		\widetilde{C}_k(m) = (-1)^{mn} \sum_{\substack{J = \{j_1, \dots, j_t\} \subset \{1, \dots, n\} \\ t \geq k+3}} e_k(J) \sum_{(\varepsilon_1, \dots, \varepsilon_n) \in \{\pm 1\}^n} \varepsilon_{j_1} \cdots \varepsilon_{j_t} \prod_{j=1}^n e^{2\pi i \frac{\varepsilon_j m}{2p_j}}.
	\end{align*}
	For each $J$, the most inner sum is expressed as
	\begin{align*}
		\sum_{(\varepsilon_1, \dots, \varepsilon_n) \in \{\pm 1\}^n} \varepsilon_{j_1} \cdots \varepsilon_{j_t} \prod_{j=1}^n e^{2\pi i \frac{\varepsilon_j m}{2p_j}} &= \prod_{j \in J} (e^{\pi i \frac{m}{p_j}} - e^{-\pi i \frac{m}{p_j}}) \prod_{j \not\in J} (e^{\pi i \frac{m}{p_j}} + e^{-\pi i \frac{m}{p_j}})\\
		&= 2^n i^t \prod_{j \in J} \sin \left(\frac{\pi m}{p_j} \right) \prod_{j \not\in J} \cos \left(\frac{\pi m}{p_j} \right).
	\end{align*}
	Since $\lvert J\rvert = t \geq k+3$, if $m$ is divisible by at least $n-k-2$ of the $p_j$'s, there exists a $j \in J$ such that $m \equiv 0 \pmod{p_j}$. Thus each term of $\widetilde{C}_k (m)$ equals $0$.
\end{proof}

In the particular case of $k=n-3$, $\widetilde{C}_k(m)$ satisfies a simple product formula.

\begin{cor}\label{CK-sinproduct}
	\[
		\widetilde{C}_{n-3} (m) = \frac{(-1)^{mn} (2i)^n}{(2P)^{n-3} (n-3)!} \prod_{j=1}^n \sin \left(\frac{\pi m}{p_j} \right).
	\]
\end{cor}

\begin{proof}
	For $k = n-3$, we have
	\begin{align*}
		\widetilde{C}_{n-3}(m) &= \frac{(-1)^{mn}}{(2P)^{n-3} (n-3)!} \sum_{(\varepsilon_1, \dots, \varepsilon_n) \in \{\pm 1\}^n} \varepsilon_1 \cdots \varepsilon_n \prod_{j=1}^n e^{2\pi i \frac{\varepsilon_j m}{2p_j}}.
	\end{align*}
	As we see above, the sum over $\bm{\varepsilon}$ equals the product of sine functions.
\end{proof}

\subsection{$S$-transformations of Homological blocks}
\label{sec:Psi-S}

By \eqref{ep-sum-vanish}, we can apply \cref{false-theta-general-k} for $\bm{a} = (\varepsilon_1 \cdots \varepsilon_n c_{\bm{\varepsilon}}(k))_{\bm{\varepsilon}}$ and $\bm{\mu} = (\ell(\bm{p}, \bm{\varepsilon})/\sqrt{2P})_{\bm{\varepsilon}}$ to the expression given in \cref{Psi-P=theta-tilde},
\[
	\widehat{\Psi}(\tau) = \frac{1}{2} \sum_{(\varepsilon_1, \dots, \varepsilon_n) \in \{\pm 1\}^n} \varepsilon_1 \cdots \varepsilon_n \sum_{k=0}^{n-3} (2P)^{k/2} c_{\bm{\varepsilon}}(k) \widetilde{\theta}_{k, \frac{\ell}{\sqrt{2P}}}(\tau).
\]
Therefore, we obtain 
\begin{align*}
	\widehat{\Psi} \left(-\frac{1}{\tau}\right)
	&= -\frac{1}{2} \sgn(\Re(\tau)) \sum_{k=0}^{n-3} \frac{(-1)^\iota (-i)^{1/2} (2P)^{\frac{k-1}{2}} \tau^{\kappa+\iota+\frac{1}{2}}}{(2\pi i)^\kappa} \sum_{r=0}^\kappa d_{\kappa,\iota,r} \tau^r\\
	&\qquad \times \sum_{\nu =0}^{2P-1} \sum_{(\varepsilon_1, \dots, \varepsilon_n) \in \{\pm 1\}^n} \varepsilon_1 \cdots \varepsilon_n c_{\bm{\varepsilon}}(k) e^{2\pi i \frac{\ell \nu}{2P}} \widetilde{\theta}_{2r+\iota, \frac{\nu}{\sqrt{2P}}}(\tau)\\
	&\quad -\frac{i}{2} \sum_{k=0}^{n-3} (2P)^{k/2} \int_0^{i\infty} \frac{\sum_{(\varepsilon_1, \dots, \varepsilon_n) \in \{\pm 1\}^n} \varepsilon_1 \cdots \varepsilon_n c_{\bm{\varepsilon}}(k) \theta_{k+1, \frac{\ell}{\sqrt{2P}}}(z)}{\sqrt{-i(z+1/\tau)}} dz,
\end{align*}
where $(\kappa, \iota)$ is defined for each $k$ by $k = 2\kappa + \iota$ as in \cref{false-theta-general-k}. By putting
\begin{align*}
	\widetilde{\eta}_{k,r}(\tau) &= -\frac{1}{2} \sum_{\nu =0}^{2P-1} \sum_{(\varepsilon_1, \dots, \varepsilon_n) \in \{\pm 1\}^n} \varepsilon_1 \cdots \varepsilon_n c_{\bm{\varepsilon}}(k) e^{2\pi i \frac{\ell \nu}{2P}} \widetilde{\theta}_{2r+\iota, \frac{\nu}{\sqrt{2P}}}(\tau),\\
	\eta_k(\tau) &= -\frac{1}{2} \sum_{(\varepsilon_1, \dots, \varepsilon_n) \in \{\pm 1\}^n} \varepsilon_1 \cdots \varepsilon_n c_{\bm{\varepsilon}}(k) \theta_{k, \frac{\ell}{\sqrt{2P}}}(\tau),
\end{align*}
we obtain the $S$-transformation formula of $\widehat{\Psi}(\tau)$ as follows.

\begin{thm}\label{thm:S-trans-WRT}
	The function $\widehat{\Psi}(\tau)$ satisfies
	\begin{align*}
		\widehat{\Psi} \left(-\frac{1}{\tau}\right) &= \sgn(\Re(\tau)) \sum_{k=0}^{n-3} \frac{(-1)^\iota (-i)^{1/2} (2P)^{\frac{k-1}{2}} \tau^{\kappa+\iota+\frac{1}{2}}}{(2\pi i)^\kappa} \sum_{r=0}^\kappa d_{\kappa,\iota,r} \tau^r \widetilde{\eta}_{k,r}(\tau)\\
		&\quad +i \sum_{k=0}^{n-3} (2P)^{k/2} \int_0^{i\infty} \frac{\eta_{k+1}(z)}{\sqrt{-i(z+1/\tau)}} dz,
	\end{align*}
	where $k=2\kappa+\iota$ as above.
\end{thm}

\subsection{Witten's asymptotic conjecture}
\label{sec:Psi-asy}

Let $K$ be a positive integer. By \cref{thm:q-series-and-quantum-invariant} and \cref{Psi-P=theta-tilde},
\begin{align}\label{Psihat-WRT}
\begin{split}
	\widehat{\Psi}(1/K) &:= \lim_{t \to 0+} \widehat{\Psi}(1/K + it)\\ 
	&= (-1)^n 2 \xi_K^{\frac{1}{4} \Theta_0} (\xi_K^{1/2} - \xi_K^{-1/2}) \tau_K - \mathcal{P}(\xi_K),
\end{split}
\end{align}
where $\xi_K = e^{2\pi i/K}$. On the other hand, we can also compute the limit
\[
	\lim_{t \to 0+} \widehat{\Psi}(1/K + it) = \lim_{t \to 0+} \widehat{\Psi} \left(-\frac{1}{-K + it} \right)
\]
by using \cref{thm:S-trans-WRT}. Based on this idea, we show the following formula of the WRT invariant.

\begin{thm}\label{thm:asymp-Andersen}
	The notations are the same as before. Then the asymptotic expansion in the limit $K \to \infty$ is given by
	\begin{align*}
		\sqrt{2/K} \xi_K^{\frac{1}{4} \Theta_0} \sin(\pi/K) \tau_K &\sim \sum_{m=1}^{2P} e^{2\pi i \frac{-m^2 K}{4P}} \sum_{k=0}^{n-3} \sum_{r=0}^\kappa  \alpha_{k,r}(m) K^{\kappa+\iota+r}\\
		&\quad + \sum_{r=0}^\infty \beta_{r} K^{-r-1/2}
		+
		\frac{(-1)^n}{2\sqrt{2} i} K^{-1/2} \mathcal{P}(\xi_K),
	\end{align*}
	where $k = 2\kappa + \iota$ as before and the coefficients are defined by
	\begin{align*}
		\alpha_{k,r}(m) 
		&= \frac{(-1)^{\kappa+r+n+1} i^{-1/2} (2P)^{\frac{k+2r+\iota-1}{2}} }{2^{5/2} (4\pi i)^{\kappa-r} (2r+\iota+1)!} \frac{k!}{(\kappa-r)!} \\
		&\quad \times \bigg(\widetilde{C}_k(m) - (-1)^k \widetilde{C}_k(-m) \bigg) B_{2r+\iota+1} \left(\frac{m}{2P}\right),\\
		\beta_{r} &= (-1)^n \frac{(2r-1)!! (-i)^r}{2^{r+3/2} r!} \sum_{k=0}^{n-3} (2P)^{k/2} i \int_0^{\infty} \frac{\eta_{k+1}(iy)}{y^{r+1/2}} dy.
	\end{align*}
	Here the integral in the definition of $\beta_{r}$ converges since $\eta_{k+1}(iy)$ decays exponentially as $y \to 0$ and $y \to \infty$. 
\end{thm}

We remark that the set
\[
\left\{ \left[ \frac{-m^2}{4P} \right] \in \mathbb{Q}/\mathbb{Z} \mid \text{$m$ is divisible by at most $n-3$ of the $p_j$'s} \right\}
\]
is the set of non-zero Chern-Simons invaraints~\cite{AndersenMistegard2018}.
Since $\alpha_{k,r}(m) = 0$ for $m$ that is divisible by at least $n-2$ of the $p_j$'s by \cref{CK-sincosprod}, this means that we solve explicitly $G = \SL(2,\mathbb{C})$ version of the following asymptotic expansion conjecture by Andersen.

\begin{conj}[{Asymptotic Expansion Conjecture~\cite[Conjecture 1.1]{Andersen2013}}]
	By $\mathcal{Z}_G^{(K)}$ we will denote the Reshetikhin-Turaev TQFT at level $K$ for the semisimple, simply connected Lie group $G$. There exist constants (depend on an oriented compact three manifold $X$) $d_{j,r} \in \mathbb{Q}$ and $b_{j,r} \in \mathbb{C}$ for $r=1,\dots, u_j$, $j=0,1,\dots, n$ and $a_{j,r}^l \in \mathbb{C}$ for $j=0, 1, \dots, n$, $l = 1, 2, \dots$ such that the asymptotic expansion of $\mathcal{Z}_G^{(K)}(X)$ in the limit $K \to \infty$ is given by
	\[
		\mathcal{Z}_G^{(K)}(X) \sim \sum_{j=0}^n e^{2\pi i K q_j} \sum_{r=1}^{u_j} K^{d_{j,r}} b_{j,r} \left(1 + \sum_{l=1}^\infty a_{j,r}^l K^{-l} \right),
	\]
	where $q_0 = 0, q_1, \dots, q_n$ are the finitely many different values of the Chern-Simons functional on the space of flat $G$-connections on $X$.
\end{conj}

For the proof of this theorem, we prepare the following lemma.

\begin{lem}
	For $k=2\kappa+\iota$ as before,
	\begin{align*}
		&\lim_{t \to 0+} \widetilde{\eta}_{k,r}(-K+it)\\
		&= \frac{(2P)^{r+\iota/2}}{2(2r+\iota+1)} \sum_{m=1}^{2P} \bigg(C_{k,K}(m) - (-1)^k C_{k,K}(-m) \bigg) B_{2r+\iota+1} \left(\frac{m}{2P}\right).
	\end{align*}
\end{lem}

\begin{proof}
	By the definitions,
	\begin{align*}
		\widetilde{\eta}_{k,r}(-K+it) &= -\frac{1}{2 (2P)^{r+\iota/2}} \sum_{m \in \mathbb{Z}} \sgn(m) C_{k,K}(m) m^{2r+\iota}  e^{- m^2 \frac{\pi t}{2P}}\\
			&= -\frac{1}{2 (2P)^{r+\iota/2}} \sum_{m=1}^\infty \bigg(C_{k,K}(m) - (-1)^\iota C_{k,K}(-m) \bigg) m^{2r+\iota} e^{-m^2 \frac{\pi t}{2P}}.
	\end{align*}
	Applying \cref{lem:asymp-L} and \eqref{eq:L-val} to this equation yields the limit formula.
\end{proof}

\begin{proof}[Proof of \cref{thm:asymp-Andersen}]
	By \cref{thm:S-trans-WRT} and the above lemma, we have
	\begin{align*}
		&\widehat{\Psi} \left(\frac{1}{K} \right) = \lim_{t \to 0+} \widehat{\Psi} \left(-\frac{1}{-K + it} \right)\\
		&= \sum_{k=0}^{n-3} \sum_{r=0}^\kappa \frac{(-1)^{\kappa+r+1} i^{1/2} (2P)^{\frac{k+2r+\iota-1}{2}} }{2 (4\pi i)^{\kappa-r} (2r+\iota+1)!} \frac{k!}{(\kappa-r)!} K^{\kappa+\iota+r+\frac{1}{2}}\\
		&\qquad \times \sum_{m=1}^{2P} e^{2\pi i \frac{-m^2 K}{4P}} \bigg(\widetilde{C}_k(m) - (-1)^k \widetilde{C}_k(-m) \bigg) B_{2r+\iota+1} \left(\frac{m}{2P}\right)\\
		&\quad - \sum_{k=0}^{n-3} (2P)^{k/2} \int_0^{\infty} \frac{\eta_{k+1}(iy)}{\sqrt{-i(iy-1/K)}} dy.
	\end{align*}
	As for the integral part, by Taylor's theorem, for any $N \in \mathbb{Z}_{\geq 0}$ and a fixed $y > 0$, we have
	\begin{align*}
		\left\lvert \frac{1}{\sqrt{-i(iy-1/K)}} - \sum_{r=0}^N \frac{(2r-1)!! (-i)^r}{2^r y^{r+1/2}} \frac{K^{-r}}{r!} \right\rvert \leq C_N \frac{K^{-N-1}}{y^{N+3/2}}
	\end{align*}
	for some constant depending on $N$. Then, since
	\begin{align*}
		&\left\lvert \int_0^{\infty} \frac{\eta_{k+1}(iy)}{\sqrt{-i(iy-1/K)}} dy - \sum_{r=0}^N \frac{(2r-1)!! (-i)^r}{2^r r!} \int_0^{\infty} \frac{\eta_{k+1}(iy)}{y^{r+1/2}} dy \cdot K^{-r} \right\rvert\\
		&\leq C_N \int_0^\infty \frac{\lvert \eta_{k+1}(iy) \rvert}{y^{N+\frac{3}{2}}} dy \cdot K^{-N-1},
	\end{align*}
	we obtain
	\begin{align*}
		\int_0^{\infty} \frac{\eta_{k+1}(iy)}{\sqrt{-i(iy-1/K)}} dy \sim \sum_{r=0}^\infty \frac{(2r-1)!! (-i)^r}{2^r r!} \int_0^{\infty} \frac{\eta_{k+1}(iy)}{y^{r+1/2}} dy \cdot K^{-r}.
	\end{align*}
	In conclusion, \cref{thm:asymp-Andersen} follows from \eqref{Psihat-WRT}.
\end{proof}

As a corollary, the leading term of the asymptotic expansion gives the asymptotic formula for the WRT invariant.
\begin{cor}\label{Hikami}
	The asymptotic formula of the WRT invariant $\tau_K$ in a limit $K \to \infty$ is given by
	\begin{align*}
		\sqrt{2/K} \sin (\pi/K) \tau_K &\sim K^{n-3} \frac{2^{n-2}}{(n-2)! \sqrt{P}} \xi_K^{-\frac{1}{4} \Theta_0} e^{-\frac{2n-3}{4} \pi i} \\
		&\quad \times \sum_{m=1}^{2P} (-1)^{mn} e^{2\pi i \frac{-m^2K}{4P}} B_{n-2} \left(\frac{m}{2P} \right) \prod_{j=1}^n \sin \left(\frac{\pi m}{p_j} \right).
	\end{align*}
\end{cor}

\begin{proof}
	The right-hand side of \cref{thm:asymp-Andersen} is asymptotically equivalent to the term with $K^{n-3}$, that is,
	\begin{align*}
		\sqrt{2/K} \xi_K^{\frac{1}{4} \Theta_0} \sin(\pi/K) \tau_K \sim \sum_{m=1}^{2P} e^{2\pi i \frac{-m^2 K}{4P}} \alpha_{n-3, \lfloor \frac{n-3}{2} \rfloor} (m) K^{n-3}.
	\end{align*}
	Since $\widetilde{C}_{n-3}(-m) = (-1)^n \widetilde{C}_{n-3}(m)$ and \cref{CK-sinproduct}, we have
	\begin{align*}
		\alpha_{n-3,\lfloor \frac{n-3}{2} \rfloor}(m) 
		&= \frac{2^{n-2}}{(n-2)! \sqrt{P}} e^{-\frac{2n-3}{4}\pi i} (-1)^{mn}B_{n-2} \left(\frac{m}{2P}\right) \prod_{j=1}^n \sin \left(\frac{\pi m}{p_j} \right),
	\end{align*}
	which concludes the proof.
\end{proof}

We remark that the formula in \cref{Hikami} coincides with a formula given by Hikami in~\cite[Proposition 4]{Hikami2006} based on Lawrence-Rozansky~\cite{LawrenceRozansky1999}. 

Finally, we rewrite \cref{Hikami} in terms of Chern-Simons invariants for flat connections of the Seifert manifolds. Hereafter, following the notations used in~\cite{AndersenMistegard2018}, we assume that $p_2, \dots, p_n$ in $(p_1, p_2, \dots, p_n)$ are odd by taking a permutation if needed. 

\begin{thm}\label{CS}
	The asymptotic formula of the WRT invariant $\tau_K$ in a limit $K \to \infty$ is given by
\begin{align*}
	&\sqrt{2/K} \sin (\pi/K) \tau_K \\
	&\quad \sim K^{n-3} \frac{2^{n-2}}{(n-2)! \sqrt{P}} \xi_K^{-\frac{1}{4} \Theta_0} e^{-\frac{2n-3}{4} \pi i} \sum_{\bm{l} \in L}
	e^{2\pi i \mathrm{CS}(\bm{l})K}\\
			&\qquad \times (-1)^{l_1} \sin \left(\pi l_1 \frac{P}{p_1^2} \right) \prod_{j=2}^n \sin \left(2\pi l_j \frac{P}{p_j^2} \right)\\
			&\qquad \times \sum_{(\varepsilon_1, \dots, \varepsilon_n) \in \{\pm 1\}^n} \varepsilon_1 \cdots \varepsilon_n \overline{B}_{n-2} \left(\frac{\varepsilon_1 l_1}{2p_1} + \sum_{j=2}^n \frac{\varepsilon_j l_j}{p_j} \right),
\end{align*}
where $L = L(p_1,\dots,p_n)$ is a subset of
\[
    \left\{(l_1,\dots,l_n) \in \mathbb{Z}^n \mid 0 \leq l_1 \leq p_1, 0 \leq l_j \leq \frac{p_j-1}{2}\quad (\text{for }j=2, \dots, n)\right\}
\]
such that there are at least three non-zero components $l_j$, and $\mathrm{CS}(\bm{l})$ is the Chern-Simons invariant for $\bm{l}$. 
The set $L$ is isomorphic to the set of connected components of the moduli space of non-trivial $\SL(2,\mathbb{C})$-flat connections on the Seifert fibered integral homology sphere. 
\end{thm}

\begin{proof}
By the Chinese remainder theorem, we have $\mathbb{Z}/2p_1\mathbb{Z} \times \mathbb{Z}/p_2\mathbb{Z} \times \cdots \times \mathbb{Z}/p_n \mathbb{Z} \cong \mathbb{Z}/2P\mathbb{Z};$
\[
	(l_1, l_2, \dots, l_n) \mapsto P \left(\frac{l_1}{p_1} + \sum_{j=2}^n \frac{2l_j}{p_j} \right) \pmod{2P}.
\]
In \cref{Hikami}, setting $m=P\left(\frac{l_1}{p_1} + \sum_{j=2}^n \frac{2l_j}{p_j} \right)$ gives
\begin{align*}
	&\sqrt{2/K} \sin (\pi/K) \tau_K\\
	&\quad \sim K^{n-3} \frac{2^{n-2}}{(n-2)! \sqrt{P}} \xi_K^{-\frac{1}{4} \Theta_0} e^{-\frac{2n-3}{4} \pi i} \sum_{l_1=0}^{2p_1-1} \sum_{l_2 =0}^{p_2 -1} \cdots \sum_{l_n =0}^{p_n-1} e^{2\pi i \frac{- \left(P \left(\frac{l_1}{p_1} + \sum_{j=2}^n \frac{2l_j}{p_j} \right)\right)^2K}{4P}}\\
			&\qquad \times (-1)^{l_1} \overline{B}_{n-2} \left(\frac{l_1}{2p_1} + \sum_{j=2}^n \frac{l_j}{p_j} \right) \sin \left(\pi l_1 \frac{P}{p_1^2} \right) \prod_{j=2}^n \sin \left(2\pi l_j \frac{P}{p_j^2} \right),
\end{align*}
where $\overline{B}_n(x) = B_n(x-\lfloor x \rfloor)$. Finally, we use the fact that for each $(l_1, l_2, \dots, l_n) \in \mathbb{Z}/2p_1\mathbb{Z} \times \mathbb{Z}/p_2 \mathbb{Z} \times \cdots \times \mathbb{Z}/p_n \mathbb{Z}$, the value
	\[
		\left(P \left(\frac{\varepsilon_1 l_1}{p_1} + \sum_{j=2}^n \frac{2\varepsilon_j l_j}{p_j} \right) \right)^2 \pmod{4P}
	\]
	is independent of the choice of $\bm{\varepsilon} = (\varepsilon_1, \varepsilon_2, \dots, \varepsilon_n) \in \{\pm 1\}^n$.
	This fact can be easily checked as follows: Fix $(l_1, \dots, l_n)$ and take any $\bm{\varepsilon}, \bm{\varepsilon}' \in \{\pm 1\}^n$. Then we have
	\begin{align*}
		&\left(P \left(\frac{\varepsilon_1 l_1}{p_1} + \sum_{j=2}^n \frac{2\varepsilon_j l_j}{p_j} \right) \right)^2 - \left(P \left(\frac{\varepsilon'_1 l_1}{p_1} + \sum_{j=2}^n \frac{2\varepsilon'_j l_j}{p_j} \right) \right)^2\\
		&\quad = P^2 \left(\frac{(\varepsilon_1 + \varepsilon'_1) l_1}{p_1} + 2\sum_{\substack{2 \leq j \leq n \\ \varepsilon_j = \varepsilon'_j}} \frac{2\varepsilon_j l_j}{p_j} \right) \left(\frac{(\varepsilon_1 - \varepsilon'_1) l_1}{p_1} + 2\sum_{\substack{2 \leq j \leq n \\ \varepsilon_j \neq \varepsilon'_j}} \frac{2\varepsilon_j l_j}{p_j} \right) \equiv 0
	\end{align*}
	$\mod{4P}$. Therefore, we have
	\begin{align*}
	&\sqrt{2/K} \sin (\pi/K) \tau_K\\
	&\quad \sim K^{n-3} \frac{2^{n-2}}{(n-2)! \sqrt{P}} \xi_K^{-\frac{1}{4} \Theta_0} e^{-\frac{2n-3}{4} \pi i} \sum_{l_1=1}^{p_1-1} \sum_{l_2 =1}^{\frac{p_2 -1}{2}} \cdots \sum_{l_n =1}^{\frac{p_n-1}{2}} e^{2\pi i \frac{-\left(P \left(\frac{l_1}{p_1} + \sum_{j=2}^n \frac{2l_j}{p_j} \right)\right)^2 K}{4P}}\\
			&\qquad \times (-1)^{l_1} \sum_{(\varepsilon_1, \dots, \varepsilon_n) \in \{\pm 1\}^n} \overline{B}_{n-2} \left(\frac{\varepsilon_1 l_1}{2p_1} + \sum_{j=2}^n \frac{\varepsilon_j l_j}{p_j} \right)\\
			&\qquad \times \varepsilon_1 \cdots \varepsilon_n \sin \left(\pi l_1 \frac{P}{p_1^2} \right) \prod_{j=2}^n \sin \left(2\pi l_j \frac{P}{p_j^2} \right).
\end{align*}
	We remark that the terms corresponding to $l_j = 0$ or $l_1 = p_1$ do not appear due to the $\sin$-terms. Comparing Section 1 in~\cite{AndersenMistegard2018}, we complete the proof.
\end{proof}

\subsection{Examples}
\label{sec:Examples}

\begin{ex}
	For $(p_1, p_2, p_3) = (2,3,5)$,
	\begin{align*}
	&\xi_K^{\frac{181}{120}} (\xi_K^{1/2} - \xi_K^{-1/2})\tau_K \\
	&\quad \sim \frac{1}{\sqrt{5}} \sqrt{\frac{K}{i}} \sum_{l_3 =1}^{2} e^{2\pi i \frac{- \left(30 \left(\frac{1}{2} + \frac{2}{3} + \frac{2l_3}{5}\right)\right)^2K}{120}} \sin \left(2\pi l_3 \frac{6}{5} \right)\\
	&\qquad \times \sum_{(\varepsilon_1, \varepsilon_2, \varepsilon_3) \in \{\pm 1\}^3} \varepsilon_1 \varepsilon_2 \varepsilon_3 \overline{B}_1 \left(\frac{\varepsilon_1}{4} + \frac{\varepsilon_2}{3} + \frac{\varepsilon_3 l_3}{5} \right)\\
	&\quad = \frac{2}{\sqrt{5}} \sqrt{\frac{K}{i}} \left(e^{-\frac{1}{60}\pi iK} \sin \left(\frac{\pi}{5} \right) + e^{-\frac{49}{60} \pi i K} \sin \left(\frac{2\pi}{5} \right) \right),
\end{align*}
which coincides with the leading term in~\cite[(18)]{LawrenceZagier1999}.
\end{ex}

\begin{ex}
	For $(p_1, p_2, p_3) = (4,3,5)$,
	\begin{align*}
	&\xi_K^{\frac{49}{240}} (\xi_K^{1/2} - \xi_K^{-1/2}) \tau_K\\
	&\quad \sim \frac{2}{\sqrt{30}} \sqrt{\frac{K}{i}} \sum_{l_1=1}^{3} \sum_{l_2 =1}^{1} \sum_{l_3 =1}^{2} e^{2\pi i \frac{- \left(60 \left(\frac{l_1}{4} + \frac{2l_2}{3} + \frac{2l_3}{5} \right)\right)^2K}{240}}\\
			&\qquad \times (-1)^{l_1} \sin \left(\pi l_1 \frac{15}{4} \right) \sin \left(2\pi l_2 \frac{20}{3} \right) \sin \left(2\pi l_3 \frac{12}{5} \right)\\
			&\qquad \times \sum_{(\varepsilon_1, \varepsilon_2, \varepsilon_3) \in \{\pm 1\}^3} \varepsilon_1 \varepsilon_2 \varepsilon_3 \overline{B}_1 \left(\frac{\varepsilon_1 l_1}{8} +\frac{\varepsilon_2 l_2}{3} + \frac{\varepsilon_3 l_3}{5} \right)\\
	&\quad = \frac{1}{\sqrt{5}} \sqrt{\frac{K}{i}} \bigg(-\sin \left(\frac{\pi}{5} \right) e^{-\frac{121}{120} \pi iK} + \sin \left(\frac{2\pi}{5} \right) e^{-\frac{49}{120} \pi iK}\\
	&\qquad - \sqrt{2} \sin \left(\frac{2\pi}{5} \right) e^{-\frac{1}{30}\pi iK} + \sqrt{2} \sin \left(\frac{\pi}{5} \right) e^{-\frac{49}{30} \pi iK} \bigg).
\end{align*}
The coefficients of $e^{-\frac{1}{120} \pi i K}$ and $e^{-\frac{169}{120} \pi iK}$ are $0$. This result coincides with the asymptotic formula given in~\cite[p.677]{Hikami2005IJM}.
\end{ex}

In the above example, some coefficients of $e^{2\pi i \mathrm{CS}(l)K}$ vanish. For $n=3$, we can characterize when the coefficients vanish as follows.

\begin{ex}
In \cref{CS}, the terms for $(l_1, l_2, l_3) \in L(p_1,p_2,p_3)$ with $l_1 = p_1$ vanish due to the sin-terms. Thus, we will consider the subset
\begin{align*}
	L' = \left\{(l_1, l_2, l_3) \in \mathbb{Z}^3 \mid 1 \leq l_1 \leq p_1-1, 1 \leq l_2 \leq \frac{p_2-1}{2}, 1 \leq l_3 \leq \frac{p_3-1}{2}\right\}
\end{align*}
of $L(p_1,p_2,p_3)$ below.
The vanishing condition is controlled by the sum of Bernoulli polynomials. In the particular case of $n=3$, we have the following simple formula.
\begin{align*}
	&\sum_{(\varepsilon_1, \varepsilon_2, \varepsilon_3) \in \{\pm 1\}^3} \varepsilon_1 \varepsilon_2 \varepsilon_3 \overline{B}_1 \left(\frac{\varepsilon_1 l_1}{2p_1} + \frac{\varepsilon_2 l_2}{p_2} + \frac{\varepsilon_3 l_3}{p_3} \right)\\
	&= -\sum_{(\varepsilon_1, \varepsilon_2, \varepsilon_3) \in \{\pm 1\}^3} \varepsilon_1 \varepsilon_2 \varepsilon_3 \left\lfloor \frac{\varepsilon_1 l_1}{2p_1} + \frac{\varepsilon_2 l_2}{p_2} + \frac{\varepsilon_3 l_3}{p_3} \right\rfloor = \begin{cases}
			2 &\text{if } (\frac{l_1}{2p_1}, \frac{l_2}{p_2}, \frac{l_3}{p_3}) \in T_{1/2},\\
			0 &\text{if } (\frac{l_1}{2p_1}, \frac{l_2}{p_2}, \frac{l_3}{p_3}) \not\in T_{1/2},
		\end{cases}
\end{align*}
where $T_{1/2}$ is an inner region of the tetrahedron with four vertices $(0,0,0)$, $(1/2,1/2,0)$, $(1/2,0,1/2)$, $(0,1/2,1/2)$, that is,
\[
	T_{1/2} = \{(x,y,z) \in [0,1/2]^3 \mid -z < x-y < z, z < x+y < 1-z\}.
\]
Here we use the fact that $(\frac{l_1}{2p_1}, \frac{l_2}{p_2}, \frac{l_3}{p_3})$ is never on the boundary $\partial T_{1/2}$ for $\bm{l} \in L'$. By the expression
\begin{align*}
	\mathrm{CS}(l_1, l_2, l_3) = - \frac{1}{4P} \left(P \left(\frac{l_1}{p_1} + \frac{2l_2}{p_2} + \frac{2l_3}{p_3} \right)\right)^2 \pmod{1},
\end{align*}
we have
\begin{align*}
	\sqrt{2/K} \sin (\pi/K) \tau_K &\sim \frac{-4 \sqrt{i}}{\sqrt{P}} \xi_K^{-\frac{1}{4} \Theta_0} \sum_{\substack{\bm{l}=(l_1, l_2, l_3) \in L' \\ (\frac{l_1}{2p_1}, \frac{l_2}{p_2}, \frac{l_3}{p_3}) \in T_{1/2}}} e^{2\pi i \mathrm{CS}(\bm{l}) K} (-1)^{l_1}\\
	&\quad \times \sin \left(\pi l_1 \frac{P}{p_1^2} \right) \sin \left(2\pi l_2 \frac{P}{p_2^2} \right) \sin \left(2\pi l_3 \frac{P}{p_3^2} \right).
\end{align*}
\end{ex}

We remark that the square of the product of sin-terms in this formula is the adjoint Reidemeister torsion in \cite{Freed1992}.


\subsection*{Acknowledgment}
 We would like to thank Hiroyuki Fuji, Kazuhiro Hikami, Kohei Iwaki, Takahiro Kitayama, Nobushige Kurokawa, Akihito Mori, Hitoshi Murakami, Yuya Murakami and Yoshikazu Yamaguchi for valuable discussions. This work is partially supported by
 JSPS KAKENHI Grant Number JP20K14292, JP21K18141, JP17K05243, JP21K03240 and by JST CREST Grant Number JPMJCR14D6.


\bibliographystyle{amsplain}
\bibliography{references} 


\end{document}